\documentclass[a4paper,11pt]{article}

\usepackage[notext]{kpfonts}
\usepackage{baskervald}

\usepackage{amsfonts,amssymb,amsmath,amsthm,abstract,color}
\usepackage[mathscr]{euscript}
\usepackage[ps,all,arc,rotate]{xy}
\usepackage[lmargin=1in,rmargin=1in,tmargin=1in,bmargin=1in]{geometry}
\usepackage{fancyhdr}
\usepackage{dsfont}
\usepackage{pb-diagram}
\usepackage{enumitem}
\usepackage[backref=page]{hyperref}
\usepackage{soul}
\usepackage[usenames,dvipsnames,natural,table]{xcolor}
\usepackage{booktabs,tabularx,longtable}
\usepackage{tikz-cd}
\usepackage{extpfeil}
\usepackage{colortbl}
\usepackage{float}
\hypersetup{colorlinks=true,citecolor=NavyBlue,linkcolor=Brown,urlcolor=Orange}

\usepackage{stmaryrd}
\usepackage[capitalise]{cleveref}
\numberwithin{equation}{section}

\setstcolor{red}

\usepackage{caption}

\usepackage{titlesec}

\titlespacing*{\section}
{0pt}{3ex plus 1ex minus .2ex}{2.5ex plus .2ex}
\titlespacing*{\subsection}
{0pt}{2ex plus 1ex minus .2ex}{1.5ex plus .2ex}
\titlespacing*{\subsubsection}
{0pt}{2ex plus 1ex minus .2ex}{1.5ex}

\newcommand{\sectionpoint}{\if \@empty\titlesec \else}              

\newcommand*{\justifyheading}{\raggedright}
\titleformat{\section}{\normalfont \Large \bfseries \justifyheading}{\thesection.}{0.5em}{}
\titleformat{\subsection}[runin]{\normalfont \large \bfseries \justifyheading}{\thesubsection.}{0.5em}{}[.]
\titleformat{\subsubsection}[runin]{\normalfont \normalsize \bfseries \justifyheading}{\thesubsubsection.}{0.5em}{}[  \protect{\rule[3pt]{10pt}{0.5pt}}]

\theoremstyle{plain}
\newtheorem*{theorem*}{Theorem}
\newtheorem{theorem}{Theorem}[section]

\theoremstyle{definition}

\theoremstyle{remark}

\newsavebox{\pmatrixbox}
\newenvironment{colorpmatrix}
  {\begin{lrbox}{\pmatrixbox}
   \mathsurround=0pt
   $\displaystyle
   \begin{pmatrix}}
  {\end{pmatrix}$%
   \end{lrbox}%
   \usebox{\pmatrixbox}%
   \kern-\wd\pmatrixbox
   \makebox[0pt][l]{$\left(\vphantom{\usebox{\pmatrixbox}}\right.$}%
   \kern\wd\pmatrixbox
}
\DeclareMathOperator{\id}{id}

\DeclareMathOperator{\rk}{rk}

\DeclareMathOperator{\bO}{\mathrm{O}}
\DeclareMathOperator{\Sgn}{sgn}
\DeclareMathOperator{\A}{\mathbf{A}}
\DeclareMathOperator{\bD}{\mathbf{D}}
\DeclareMathOperator{\bL}{\mathbf{L}}

\DeclareMathOperator{\bH}{\mathbf{H}}
\DeclareMathOperator{\bK}{\mathbf{K}}
\DeclareMathOperator{\U}{\mathbf{U}}
\DeclareMathOperator{\bLambda}{\boldsymbol{\Lambda}}
\DeclareMathOperator{\OG6}{OG6}

\begin{document}

\title{\textbf{Corrigendum to ``Nonsymplectic automorphisms of prime order on O’Grady’s sixfolds''}}

\author{Annalisa Grossi and Stevell Muller}
\date{ }
\maketitle

\begin{abstract}
In this short note we correct the lattice theoretic classification in \cite{Gr_OG6} of the first named author, including some missing cases and explaining how to recover them.
\end{abstract}



\noindent{\textbf{Correction of: Rev. Mat. Iberoam. (2022) 38(4):1199--1218} https://doi.org/10.4171/RMI/1341}\\

\noindent{\textbf{Published in: Rev. Mat. Iberoam. (2026) 42(4):1589--1598} https://doi.org/10.4171/RMI/1640}

\renewcommand{\thefootnote}{}
\null\footnotetext{A.G. was partially supported by the European Union - NextGenerationEU under the National Recovery and Resilience Plan (PNRR) - Mission 4 Education and research - Component 2 From research to business - Investment 1.1 Notice Prin 2022 - DD N. 104 del 2/2/2022, from title \lq\lq Symplectic varieties: their interplay with Fano manifolds and derived categories\rq\rq, proposal code 2022PEKYBJ – CUP J53D23003840006.
A.G. is member of the INdAM group GNSAGA, and was partially supported by GNSAGA.}
\renewcommand{\thefootnote}{\arabic{footnote}}

%\tableofcontents
%\newpage

\section{Introduction}

In this short note we describe all the cases which are missing in the computation of the invariant and coinvariant sublattices associated to nonsymplectic automorphisms of prime order of hyper-Kähler manifolds of OG6 type \cite{Gr_OG6}. Recall that given a hyper-Kähler manifold $X$ of OG6 type, the second integral cohomology group of $X$, equipped with its Beauville--Bogomolov--Fujiki quadratic form, is isometric as a lattice to

\[\bL := \U^{\oplus3}\oplus [-2]^{\oplus2}.\]

Any finite subgroup $G\subset \bO(\bL)$ is called \textit{effective nonsymplectic} if it arises from representing on \(\bL\) a group of nonsymplectic automorphisms of an OG6 type manifold (see \cite[\S 1.2]{Gr_OG6} for more details). Using the computer program OSCAR \cite{OSCAR-book}, we prove the following result.

\begin{theorem}
    There are exactly 45 conjugacy classes of effective nonsymplectic groups $G\subset \bO(\bL)$ of prime order.
\end{theorem}

\begin{proof}
    According to \cite[Proposition 3.4]{Gr_OG6}, it is enough to classify finite subgroups $G\subset \bO(\bL)$ of prime order whose invariant sublattice $\bL^G$ has signature $(1, \ast)$. We proceed as follows. Fix $p\in\{2,3,5,7\}$ a prime number and $\zeta$ a primitive $p$th root of unity. In order to classify, up to conjugacy, the groups $G$ as before, it is enough to determine representatives of conjugacy classes of isometries $g\in \bO(\bL)$ satisfying both of the following:
    \begin{enumerate}
        \item the invariant lattice $\bL^g$ has signature $(1,\ast)$,
        \item the real quadratic space $\ker\left((g+g^{-1})-(\zeta+\zeta^{-1})\id\right)$ has signature $(2, \ast)$.
    \end{enumerate}
    In fact, condition 2 indicates that we make $g$ (or $g^{-1}$) a distinguished generator of the associated cyclic group. We can thus apply \cite[Algorithms 1, 2, 3 \& 8]{brandhorst_hofmann}\footnote{The use of \cite[Algorithm 8]{brandhorst_hofmann} is necessary for some technical parts of Brandhorst--Hofmann machinery; namely to compute explicit generators for the image of the map $\bO(L, f)\to \bO(L^\sharp, f^\sharp)$, given an even indefinite lattice $L$ of rank larger than 2 and a finite order isometry $f\in \bO(L)$ (see \cite[\S2.1]{Gr_OG6} for the notation).}, implemented in OSCAR \cite{OSCAR} by the second named author, to obtain the results as claimed. We refer to \autoref{tab:L} for the isometry classes of $\bL_G$ and $\bL^G$, as well as the number of conjugacy classes of groups with the same associated pairs $(\bL^G, \bL_G)$, up to isometry.
\end{proof}
The theorem can also be proved by first determining a complete list of possible pairs of invariant~and coinvariant sublattices $(\bL^G, \bL_G)$ associated to effective nonsymplectic groups $G\subset \bO(\bL)$, and applying \cite[Algorithms 2, 3 \& 8]{brandhorst_hofmann} to this list to obtain the number of conjugacy classes with given invariant and coinvariant sublattices. This list of pairs $(\bL^G, \bL_G)$ was partially determined in  \cite{Gr_OG6}. In what follows, we describe the missing and erroneous cases in \textit{loc. cit.}, and we give corrected versions of \cite[Tables 1, 2, 3, 4 \& 5]{Gr_OG6}.

\section{Corrections}
We give a complete list of corrections for \cite[Tables 1, 2, 3, 4, \& 5]{Gr_OG6}. We refer to \cite[Sections 3 \& 4]{Gr_OG6} for more details on how to perform the computations needed to make such corrections, and to the notebook at
\[\text{\url{https://github.com/StevellM/Notebooks/tree/main/Corrigendum_Grossi_RMI_2022}}\]
to see how one can proceed with a computer program. For the sake of completeness, and for the reader's convenience, we reproduce in the next section Tables \ref{tab:Lambda}, \ref{tab:L_G}, \ref{tab:invariante e coinvariante azione non banale su L}, \ref{tab:H2 azione non banale su Ax} and \ref{tab:L} after corrections. In each table, the cases that are missing in \cite{Gr_OG6} appear in a {\color{blue!50!white}blue row}. In \cite[Table 3]{Gr_OG6}, an entry $(S, T)$ has not been detected as arising from an involution of $\bL$: The blue trefoil ``{\color{blue}$\clubsuit$}'' in Table~\ref{tab:invariante e coinvariante azione non banale su L}  reports this missing case. In \cite[Table 5]{Gr_OG6}, there is an erroneous entry: For clarity, we keep it in Table~\ref{tab:L} but make it appear in a {\color{red!50!white}red row}.

\begin{enumerate}

\item In the second part of \cite[Table 1]{Gr_OG6}, the case \(\left(\bLambda_G,\, \bLambda^G\right)=\left(\U^{\oplus3} \oplus [-2],\, \U\oplus [2]\right)\) is missing (see the new entry No.~3 of the second part of \autoref{tab:Lambda}). 
This correction to \cite[Table 1]{Gr_OG6} gives rise to:
\begin{itemize}
    \item the new entry No.~5 \(\left(\bLambda_G,\, S\right)= \left(\U^{\oplus3} \oplus [-2],\, \U \oplus [2] \oplus \A_3(-1)\right)\) in \autoref{tab:L_G},
    \item the new entry No.~4 \((S,\, T)=\left(\U \oplus [2] \oplus \A_3(-1),\, [4]\oplus[-2]\right)\) in \autoref{tab:invariante e coinvariante azione non banale su L}, and
    \item the new entries No.~2 in \autoref{tab:H2 azione non banale su Ax} and in the second part of \autoref{tab:L}.
\end{itemize} 
Note that the lattice \(S=\U \oplus [2] \oplus \A_3(-1) \) is already contained in \cite[Table 2]{Gr_OG6} but it is not reported in \cite[Table 3]{Gr_OG6}.

\item In \cite[Table 2]{Gr_OG6}, the case \(\left(\bLambda_G,\, S\right)= \left(\U^{\oplus3} \oplus [-2]^{\oplus2},\, \U^{\oplus2} \oplus \A_3(-1)\right)\) is missing (see the new entry No.~2 in \autoref{tab:L_G}).
This correction gives rise to:
\begin{itemize}
    \item the new entry No.~2 \((S,\, T)=\left(\U^{\oplus 2} \oplus \A_3(-1),\, [4]\right)\) in \autoref{tab:invariante e coinvariante azione non banale su L}, and
    \item the new entries No.~1 in \autoref{tab:H2 azione non banale su Ax} and in the second part of \autoref{tab:L}.
\end{itemize}

\item In \cite[Table 3]{Gr_OG6}, a possible orthogonal complement to \(S=[2]^{\oplus 2} \oplus [-2]^{\oplus2} \oplus [-4]\) in $\bL$ is given by the entry No.~3. However, there is a missing possibility, given by \(T=\U(2) \oplus [-4]\) (see the new entry No.~6 in \autoref{tab:invariante e coinvariante azione non banale su L}). 

\item In \cite[Table 3]{Gr_OG6}, the entries No.~9 and 10 are duplicates of the entry No.~4: This can be seen by remarking the following isometry of lattices: $[2]\oplus[-4] \cong [-2]\oplus[4]$. These two entries are therefore removed in \autoref{tab:invariante e coinvariante azione non banale su L}, and the new entry No.~7 has been decorated with a trefoil ``$\clubsuit$'', similarly to the entry No.~10 in \cite[Table 3]{Gr_OG6}. 

\item In \cite[Table 3]{Gr_OG6}, the entry No.~14 is missing the trefoil ``$\clubsuit$'' in the last column. Indeed, there exists an involution of $\bL$ with coinvariant and invariant sublattices respectively given by $[2]\oplus [-2]^{\oplus2}\oplus [4]$ and $[2]\oplus [-2]^{\oplus2}\oplus[-4]$ (see the new entry No.~15 in \autoref{tab:invariante e coinvariante azione non banale su L}). This correction gives rise to the new entries No.~6 in \autoref{tab:H2 azione non banale su Ax} and in the second part of  \autoref{tab:L}. 

\item In \cite[Table 3]{Gr_OG6}, the entries No.~15 and 16 are duplicates of the entries No.~14 and 13 respectively: Again, this can be checked using the identification of lattices $[2]\oplus[-4] \cong [-2]\oplus[4]$. These two entries are therefore removed in \autoref{tab:invariante e coinvariante azione non banale su L}.

\item In \cite[Table 3]{Gr_OG6}, the entries No.~19 and 20 are duplicates of the entry No.~17 and 18 respectively: We use, once again, the identification $[2]\oplus[-4] \cong [-2]\oplus[4]$ to detect those duplicates. These two entries are therefore removed in \autoref{tab:invariante e coinvariante azione non banale su L}, and the new entry No.~17 has been decorated with a trefoil ``$\clubsuit$'', similarly to the entry No.~20 in \cite[Table 3]{Gr_OG6}

\item In \cite[Table 3]{Gr_OG6}, possible orthogonal complements to \(S=\U(2) \oplus [4]\) in $\bL$ are given by the entries No.~21 and 22. However, there is a missing possibility, given by \(T=\U(2) \oplus \A_3(-1)\) (see the new entry No.~20 in \autoref{tab:invariante e coinvariante azione non banale su L}). This correction gives rise to the new entries No.~8 in \autoref{tab:H2 azione non banale su Ax} and in the second part of \autoref{tab:L}. 

\item In \cite[Table 3]{Gr_OG6}, possible orthogonal complements to \(S=[2]^{\oplus2} \oplus [-4]\) in $\bL$ are given by the entries No.~23 and 24. However, there are two missing possibilities, given by \(T=\U \oplus \A_3(-1)\) and $T = \U(2)\oplus\A_3(-1)$ (see the new respective entries No.~23 and 24 in \autoref{tab:invariante e coinvariante azione non banale su L}).

\item In \cite[Table 3]{Gr_OG6}, a possible orthogonal complement to  \(S=[2] \oplus [4]\) in $\bL$ is given by the entry No.~27. However, there is a missing possibility, given by \(T=\U  \oplus \A_3(-1)\oplus [-2]\) (see the new entry No.~28 in \autoref{tab:invariante e coinvariante azione non banale su L}). This correction gives rise to the new entries No.~11 in \autoref{tab:H2 azione non banale su Ax} and in the second part of \autoref{tab:L}.

\item In the first part of \cite[Table 5]{Gr_OG6}, the case $\left(\bL_G,\,\bL^G\right) = \left(\U \oplus [2] \oplus [-2]^{\oplus 3},\, \U(2)\right)$ is missing (see the new entry No.~3 in the first part of \autoref{tab:L}). Note that the lattice $\bL_G = \U \oplus [2] \oplus [-2]^{\oplus 3}$ is already contained in the first part of \cite[Table 5]{Gr_OG6}, but the possible orthogonal complement given by $\U(2)$ is missing. 

\item In the first part of \cite[Table 5]{Gr_OG6}, the entry No.~24 is not correct (see the colored entry No.~25 in the first part of \autoref{tab:L}). Indeed, there is no primitive extension $\left([2]^{\oplus2}\right)\oplus \left(\U(2)\oplus \bD_4(-1)\right)\subseteq \U^{\oplus3}\oplus [-2]^{\oplus2}$.

\end{enumerate}

\section{Corrected tables}

\begin{center}
	\begin{longtable}{lllllll}
		
		\caption{Pairs \((\bLambda^G,\bLambda_G)\) for \(G \subset \bO(\bLambda)\) of prime order \(p=2\) and \(\Sgn(\bLambda_G)=(2, \rk(\bLambda_G)-2)\), or \(p=2\) and \(\Sgn(\bLambda_G)=(3, \rk(\bLambda_G)-3)\), or \(p \in \{3,5,7\}\) and \(\Sgn(\bLambda_G)=(2, \rk(\bLambda_G)-2)\).}
	\label{tab:Lambda} \\
		
		\toprule
	 No. & \(|G|\) &  \(\bLambda_G\) & \(\bLambda^G\) & \(\Sgn(\bLambda_G)\) & \(a\) & \(\delta\)  \\

		\midrule
		\endfirsthead
		
		\multicolumn{7}{c}%
		{\tablename\ \thetable{}, follows from previous page} \\
		\midrule
	No. & \(|G|\) &  \(\bLambda_G\) & \(\bLambda^G\) & \(\Sgn(\bLambda_G)\) & \(a\) & \(\delta\)  \\
		\midrule
		\endhead
	
		\multicolumn{7}{c}{Continues on next page} \\
		\endfoot
		
		\bottomrule
		\endlastfoot

		\(1\) & \(2\) & \(\U^{\oplus 2} \oplus [-2 ] ^{\oplus 3 }\)& \([ 2 ] ^{\oplus 3 }\)& \((2,5)\) & \(3\) & \(1\) \\
		\(2\) & \(2\) & \(\U \oplus [ -2 ] ^{\oplus 3 } \oplus [2 ]\) & \([ 2 ] ^{\oplus 3} \oplus [ -2 ]\) &  \((2,4)\) & \(4\) & \(1\) \\
		\(3\) & \(2\)  & \(\U^{\oplus 2} \oplus [ -2 ] ^{\oplus 2 }\) & \(\U \oplus [ 2 ] ^{\oplus 2}\)&  \((2,4)\) & \(2\) & \(1\) \\
		\(4\) & \(2\)  & \([2]^{\oplus 2} \oplus [ -2 ] ^{\oplus 3}\) & \([ -2]^{\oplus 2} \oplus [2 ] ^{\oplus 3}\)  & \((2,3)\) & \(5\) & \(1\) \\
		\(5\) & \(2\)  & \(\U \oplus  [ -2 ] ^{\oplus 2 } \oplus [ 2 ]\) & \(\U \oplus [2 ] ^{\oplus 2 } \oplus [-2 ]\) &  \((2,3)\) & \(3\) & \(1\) \\
		\(6\) & \(2\) & \(\U^{\oplus 2} \oplus [ -2 ]\) & \(\U^{\oplus 2} \oplus [ 2 ]\) &  \((2,3)\) & \(1\) & \(1\) \\
		\(7\) & \(2\)  & \([ 2 ] ^{\oplus 2} \oplus [ -2 ] ^{\oplus 2}\) & \(\U \oplus [ 2 ] ^{\oplus 2} \oplus [ -2 ]^{\oplus 2}\) &  \((2,2)\) & \(4\) & \(1\) \\
		\(8\) & \(2\) & \(\U(2)^{\oplus 2}\)  & \(\U \oplus \U(2)^{\oplus 2}\)&  \((2,2)\) & \(4\) & \(0\) \\
		\(9\) & \(2\) & \(\U \oplus [ 2 ] \oplus[ -2 ]\) & \(\U^{\oplus 2} \oplus [ 2 ] \oplus [ -2 ]\) &  \((2,2)\) & \(2\) & \(1\) \\
		\(10\) & \(2\) & \(\U \oplus \U(2)\) & \(\U^{\oplus 2} \oplus \U(2)\) &  \((2,2)\) & \(2\) & \(0\) \\
		\(11\) & \(2\) & \(\U ^{\oplus 2}\)& \(\U^{\oplus 3}\) &  \((2,2)\) & \(0\) & \(0\) \\
		\(12\) & \(2\) & \([2 ]^{\oplus 2} \oplus [ -2 ]\) & \(\U^{\oplus 2} \oplus [-2]^{\oplus 2} \oplus [2]\) &  \((2,1)\) & \(3\) & \(1\) \\
	    \(13\) & \(2\) & \(\U \oplus [2]\) & \(\U^{\oplus 3} \oplus [ -2 ]\) &  \((2,1)\) & \(1\) & \(1\) \\
		\(14\) & \(2\) & \([ 2 ]^{ \oplus 2}\)& \(\U^{\oplus 3} \oplus [-2 ] ^{\oplus 2}\) &  \((2,0)\) & \(2\) & \(1\) \\
		\midrule
		\(1\) & \(2\) & \(\U^{\oplus 3} \oplus [ -2] ^{\oplus 2}\) & \([2]^{\oplus 2}\) & \((3,5)\) & \(2\) & \(1\) \\
		\(2\) & \(2\) & \(\U^{\oplus 2} \oplus [2] \oplus [-2]^{\oplus2}\) & \([2]^{\oplus2} \oplus [-2]\) & \((3,4)\) & \(3\) & \(1\) \\
        \rowcolor{blue!20!white}\(3\) & \(2\) & \(\U^{\oplus 3}\oplus [-2]\) & \(\U\oplus [2]\) & \((3,4)\) & \(1\) & \(1\)\\
		\(4\) & \(2\) & \(\U  \oplus [ 2 ] ^{\oplus 2 } \oplus [ -2] ^{\oplus 2}\) & \([ 2 ] ^{\oplus 2 } \oplus [ -2 ] ^{\oplus 2}\) &  \((3,3)\) & \(4\) & \(1\) \\
		\(5\) & \(2\) & \(\U \oplus \U(2)^{\oplus 2}\) & \(\U(2)^{\oplus2}\) &  \((3,3)\) & \(4\) & \(0\) \\
		\(6\) & \(2\) & \(\U^{\oplus 2} \oplus [ 2 ] \oplus [ -2 ]\) & \(\U  \oplus [ 2 ] \oplus [ -2 ]\) &  \((3,3)\) & \(2\) & \(1\) \\
		\(7\) & \(2\) & \(\U^{\oplus 2} \oplus \U(2)\) & \(\U \oplus \U(2)\) &  \((3,3)\) & \(2\) & \(0\) \\
    	\(8\) & \(2\) & \(\U ^{ \oplus 3}\)& \(\U^{\oplus 2}\) &  \((3,3)\) & \(0\) & \(0\) \\
		\(9\) & \(2\) & \([ -2 ] ^{\oplus 2} \oplus[ 2 ] ^{\oplus 3}\)& \([ 2 ] ^{\oplus 2} \oplus [ -2]^{\oplus 3}\) &  \((3,2)\) & \(5\) & \(1\) \\
		\(10\) & \(2\)  & \(\U \oplus [ -2 ]  \oplus [ 2 ] ^{\oplus 2}\) & \(\U \oplus [2 ] \oplus [ -2 ] ^{\oplus 2}\) &  \((3,2)\) & \(3\) & \(1\) \\
		\(11\) & \(2\)  & \(\U^{\oplus 2} \oplus [2] \)  & \(\U^{\oplus 2} \oplus [-2 ]\) &  \((3,2)\) & \(1\) & \(1\) \\
		\(12\) & \(2\) & \(\U(2)\oplus [2 ] ^{\oplus 2}\) & \(\U \oplus \U(2)\oplus [-2] ^{\oplus 2}\) &  \((3,1)\) & \(4\) & \(1\) \\
		\(13\) & \(2\)  & \(\U \oplus [ 2 ]^{\oplus 2}\)& \(\U^{\oplus 2} \oplus [ -2] ^{\oplus 2}\) &  \((3,1)\) & \(2\) & \(1\) \\
		\(14\) & \(2\) & \([ 2 ]^{\oplus 3}\)& \(\U^{\oplus 2} \oplus  [ -2 ] ^{\oplus 3}\) &  \((3,0)\) & \(3\) & \(1\) \\
		\midrule
		\(1\) & \(3\) & \(\U^{\oplus 2} \oplus \A_{2}(-1)\) & \(\U \oplus \A_{2}\) & \((2,4)\) & \(1\) &-- \\
		\(2\) & \(3\) & \(\U \oplus \U(3) \oplus \A_{2}(-1)\) & \(\U(3) \oplus \A_{2}\) & \((2,4)\) & \(3\) &-- \\
		\(3\) & \(3\) & \(\U^{\oplus 2} \) & \(\U^{\oplus3}\) & \((2,2)\) & \(0\) &-- \\
		\(4\) & \(3\) & \(\U \oplus \U(3) \) & \(\U^{\oplus 2} \oplus \U(3)\) & \((2,2)\) & \(2\) &-- \\
		\(5\) & \(3\)  & \(\A_{2}\) & \(\U^{\oplus 3} \oplus \A_{2}(-1)\) & \((2,0)\) & \(1\) & --  \\
		\midrule
		\(1\) & \(5\) & \(\U \oplus \bH_{5}\) & \(\U^{\oplus 2} \oplus \bH_{5}\) & \((2,2)\) & \(1\) & --\\
		\midrule
	    \(1\) & \(7\) & \(\U^{\oplus 2} \oplus \bK_{7}\) & \(\U \oplus \bK_{7}(-1)\) & \((2,4)\) & \(1\) &-- \\
	\end{longtable}
\end{center}

\begin{center}
	\begin{longtable}{lll}
		
		\caption{Orthogonal complements of \([4] \hookrightarrow \bLambda_G\).} 
		\label{tab:L_G}\\
		\toprule
	No. & \(\bLambda_G\) & \(S=[4]^{\perp_{\Lambda_G}}\) \\

		\midrule
		\endfirsthead
		
		\multicolumn{3}{c}%
		{\tablename\ \thetable{}, follows from previous page} \\
		\midrule
	No. & \(\bLambda_G\) & \(S=[4]^{\perp_{\Lambda_G}}\) \\
		\midrule
		\endhead

		\multicolumn{3}{c}{Continues on next page} \\
		\endfoot
		
		\bottomrule
		\endlastfoot

		\(1\) & \(\U^{\oplus3} \oplus [-2]^{\oplus2}\) & \(\U^{\oplus2} \oplus [-2]^{\oplus2} \oplus [-4]\) \\
		
        \rowcolor{blue!20!white}\(2\) & \(\U^{\oplus3} \oplus [-2]^{\oplus2}\) & \(\U^{\oplus2} \oplus \A_3(-1)\) \\
        \midrule
		\(3\) & \(\U^{\oplus2} \oplus [2] \oplus [-2]^{\oplus2}\) & \(\U \oplus [2] \oplus \A_3(-1)\) \\
		\(4\) & \(\U^{\oplus2} \oplus [2] \oplus [-2]^{\oplus2}\) & \(\U \oplus [2] \oplus [-2]^{\oplus2} \oplus [-4]\) \\
		\midrule
        \rowcolor{blue!20!white}\(5\) & \(\U^{\oplus 3}\oplus [-2]\) & $\U \oplus [2] \oplus \A_3(-1)$\\
        \midrule
		\(6\) & \(\U \oplus [2]^{\oplus2} \oplus [-2]^{\oplus2}\) & \([2]^{\oplus2} \oplus [-2]^{\oplus2} \oplus [-4]\) \\
		\(7\) & \(\U \oplus [2]^{\oplus2} \oplus [-2]^{\oplus2}\) & \(\U \oplus [-2]^{\oplus2} \oplus [4]\) \\
		\midrule
		\(8\) & \(\U \oplus \U(2)^{\oplus2}\) & \(\U(2)^{\oplus2} \oplus [-4]\) \\
		\(9\) & \(\U \oplus \U(2)^{\oplus2}\) & \(\U \oplus \U(2)\oplus [-4]\) \\
		\midrule
		\(10\) & \(\U^{\oplus2} \oplus[2] \oplus [-2]\) & \(\U  \oplus [2] \oplus [-2] \oplus [-4]\) \\
		\midrule
		\(11\) & \(\U^{\oplus 2} \oplus \U(2)\) & \(\U\oplus \U(2) \oplus [-4]\)  \\
		\(12\) & \(\U^{\oplus 2} \oplus \U(2)\) & \(\U^{\oplus2} \oplus [-4]\)  \\
		\midrule
		\(13\) & \(\U^{\oplus3}\) & \(\U^{\oplus2} \oplus [-4]\) \\
		\midrule
		\(14\) & \([2]^{\oplus3} \oplus  [-2]^{\oplus2}\) & \( [2] \oplus [-2]^{\oplus 2} \oplus  [4]\) \\
		\midrule
		\(15\) & \(\U\oplus [2]^{\oplus2}  \oplus [-2]\) & \( [2]^{\oplus2} \oplus [-2] \oplus [-4]\) \\
		\(16\) & \(\U  \oplus [2]^{\oplus2} \oplus [-2]\) & \(\U \oplus[-2] \oplus [4]\) \\
		\midrule
		\(17\) & \(\U^{\oplus2} \oplus [2]\) & \(\U \oplus[2] \oplus [-4]\) \\
		\midrule
		\(18\) & \(\U(2) \oplus [2]^{\oplus2}\) & \(\U(2) \oplus [4]\) \\
		\(19\) & \(\U(2) \oplus [2]^{\oplus 2}\) & \([2]^{\oplus2} \oplus [-4]\) \\
		\midrule
		\(20\) & \(\U \oplus [2]^{\oplus2}\) & \([2]^{\oplus2} \oplus [-4]\) \\
		\(21\) & \(\U \oplus [2]^{\oplus2}\) & \(\U \oplus [4]\) \\
		\midrule
		\(22\) & \([2]^{\oplus3}\) & \([2] \oplus [4]\) \\

	\end{longtable}	
\end{center}

\begin{center}
	\begin{longtable}{llll}
		\caption{Orthogonal complements of \(S \hookrightarrow \bL\).}
		\label{tab:invariante e coinvariante azione non banale su L}\\
	
		\toprule
		No. & \(S\) & \(T\) & \(\clubsuit\) \\

		\midrule
		\endfirsthead
		
		\multicolumn{4}{c}%
		{\tablename\ \thetable{}, follows from previous page} \\
		\midrule
		No. & \(S\) & \(T\) & \(\clubsuit\) \\
		\midrule
		\endhead
		
		\midrule
		\multicolumn{4}{c}{Continues on next page} \\
		\endfoot
		
		\bottomrule
		\endlastfoot

		\(1\) & \(\U^{\oplus2}  \oplus [-2]^{\oplus2} \oplus [-4]\) & \([4]\) & --\\
        \rowcolor{blue!20!white}\(2\) & $\U^{\oplus2}\oplus \A_3(-1)$ & \([4]\) & $\clubsuit$\\
		\(3\) & \(\U \oplus [2] \oplus [-2]^{\oplus2} \oplus [-4]\) & \([4] \oplus [-2]\) & --\\
        \rowcolor{blue!20!white}\(4\) & $\U\oplus [2]\oplus \A_3(-1)$ & \([4]\oplus[-2]\) & $\clubsuit$\\
		\(5\)  & \([2]^{\oplus2} \oplus [-2]^{\oplus2} \oplus [-4]\) & \( [-2]^{\oplus 2}\oplus[4]\) & --\\
        \rowcolor{blue!20!white}\(6\)  & \([2]^{\oplus2} \oplus [-2]^{\oplus2} \oplus [-4]\) & \(\U(2) \oplus [-4]\) & --\\
		\(7\) & \(\U \oplus [-2]^{\oplus2}  \oplus [4]\) & \([2] \oplus [-2]\oplus [-4]\) & ${\clubsuit}$\\
        \(8\) & \(\U \oplus [-2]^{\oplus2}  \oplus [4]\) & \(\U\oplus [-4]\) &--\\
		\(9\) & \(\U \oplus [-2]^{\oplus2}  \oplus [4]\) & \(\U(2) \oplus [-4]\) &--\\
		\(10\) & \(\U \oplus \U(2) \oplus [-4]\) & \(\U(2) \oplus [-4]\) & \(\clubsuit\) \\
		\(11\) & \(\U \oplus \U(2) \oplus [-4]\) & \([-2]^{\oplus2}\oplus[4]\) & -- \\
		\(12\) & \(\U^{\oplus2} \oplus [-4]\) & \( [-2]^{\oplus2}\oplus[4]\) & -- \\
		\(13\) & \(\U^{\oplus2} \oplus [-4]\) & \(\U \oplus [-4]\) & \(\clubsuit\) \\
		\(14\) & \([2] \oplus [-2]^{\oplus2} \oplus [4]\) & \(\U \oplus [-2] \oplus [-4]\) & --\\
		\(15\) &\([2] \oplus [-2]^{\oplus2} \oplus [4]\) & \([2] \oplus [-2]^{\oplus2} \oplus [-4]\) & {\color{blue}$\clubsuit$}\\
		\(16\) & \(\U \oplus[-2] \oplus [4]\) &  \([2] \oplus [-2]^{\oplus2} \oplus [-4]\) & --\\
		\(17\) & \(\U \oplus[-2] \oplus [4]\) & \(\U \oplus [-2] \oplus [-4]\) & \(\clubsuit\) \\
		\(18\) & \(\U(2) \oplus [4]\) & \(\U(2) \oplus[-2]^{\oplus2} \oplus[-4]\) & --\\
		\(19\) & \(\U(2) \oplus [4]\) & \(\U\oplus [-2]^{\oplus2} \oplus [-4]\) & --\\
       \rowcolor{blue!20!white}\(20\) & \(\U(2) \oplus [4]\) & \(\U(2)\oplus \A_3(-1)\) & $\clubsuit$\\
		\(21\) & \([2]^{\oplus2} \oplus [-4]\) & \([-2]^{\oplus4} \oplus [4]\) & --\\
		\(22\) & \([2]^{\oplus2} \oplus [-4]\) & \(\U\oplus [-2]^{\oplus2} \oplus [-4]\) & \(\clubsuit\) \\
        \rowcolor{blue!20!white}\(23\) & \([2]^{\oplus2} \oplus [-4]\) & \(\U\oplus \A_3(-1)\) & --\\
        \rowcolor{blue!20!white}\(24\) & \([2]^{\oplus2} \oplus [-4]\) & \(\U(2)\oplus \A_3(-1)\) & --\\
		\(25\) & \(\U \oplus [4]\) & \(\U \oplus[-2]^{\oplus2} \oplus [-4]\) & --\\
		\(26\) & \(\U \oplus [4]\) & \(\U \oplus \A_{3}(-1)\) & \(\clubsuit\) \\
		\(27\) & \([2] \oplus [4]\) & \(\U \oplus [-2]^{\oplus3} \oplus[-4]\) & --\\
        \rowcolor{blue!20!white}\(28\) & \([2] \oplus [4]\) & \(\U \oplus \A_3(-1) \oplus[-2]\) & $\clubsuit$\\
	\end{longtable}	
\end{center}

\begin{center}
	\begin{longtable}{llll}
			\caption{Invariant and coinvariant sublattices of nonsymplectic groups \(G \subset \bO(\bL)\) of order \(2\) and \(|G^{\sharp}|=2\) on manifolds of \(\OG6\) type.} \\
           \label{tab:H2 azione non banale su Ax} \\
	
		\toprule
		No. & \(\bL_G\) & \(\bL^G\) & example  \\

		\midrule
		\endfirsthead
		
		\multicolumn{4}{c}%
		{\tablename\ \thetable{}, follows from previous page} \\
		\midrule
		No. & \(\bL_G\) & \(\bL^G\) & example \\
		\midrule
		\endhead
		
		\midrule
		\multicolumn{4}{c}{Continues on next page} \\
		\endfoot
		
		\bottomrule
		\endlastfoot

        \rowcolor{blue!20!white}\(1\) & \(\U^{\oplus2} \oplus \A_3(-1)\) & \( [4]\) & {\tiny \(\begin{colorpmatrix} -1 & 0 & 0 & 0 & 0 & 0 & 0 & 0 \\ 0 & -1 & 0 & 0 & 0 & 0 & 0 & 0 \\ 0 & 0 & -1 & 0 & 0 & 0 & 0 & 0 \\ 0 & 0 & 0 & -1 & 0 & 0 & 0 & 0 \\ 0 & 0 & 0 & 0 & 1 & 2 & 1 & 1 \\ 0 & 0 & 0 & 0 & 2 & 1 & 1 & 1 \\ 0 & 0 & 0 & 0 & -2 & -2 & -2 & -1 \\ 0 & 0 & 0 & 0 & -2 & -2 & -1 & -2 \end{colorpmatrix}\)}  \\
       \rowcolor{blue!20!white}\(2\) & \(\U \oplus [2]\oplus \A_3(-1)\) & \( [4]\oplus [-2]\) & {\tiny \(\begin{colorpmatrix} -1 & 0 & 0 & 0 & 0 & 0 & 0 & 0 \\ 0 & -1 & 0 & 0 & 0 & 0 & 0 & 0 \\ 0 & 0 & 0 & -1 & 0 & 0 & 0 & 0 \\ 0 & 0 & -1 & 0 & 0 & 0 & 0 & 0 \\ 0 & 0 & 0 & 0 & 1 & 2 & 1 & 1 \\ 0 & 0 & 0 & 0 & 2 & 1 & 1 & 1 \\ 0 & 0 & 0 & 0 & -2 & -2 & -2 & -1 \\ 0 & 0 & 0 & 0 & -2 & -2 & -1 & -2 \end{colorpmatrix}\)}  \\
		\(3\) & \(\U \oplus \U(2) \oplus [-4]\) & \(\U(2) \oplus [-4]\) & {\tiny \(\begin{pmatrix} -1 & 0 & 0 & 0 & 0 & 0 & 0 & 0 \\ 0 & -1 & 0 & 0 & 0 & 0 & 0 & 0 \\ 0 & 0 & 0 & 0 & 1 & 0 & 0 & 0 \\ 0 & 0 & 0 & 0 & 0 & 1 & 0 & 0 \\ 0 & 0 & 1 & 0 & 0 & 0 & 0 & 0 \\ 0 & 0 & 0 & 1 & 0 & 0 & 0 & 0 \\ 0 & 0 & 0 & 0 & 0 & 0 & 0 & 1 \\ 0 & 0 & 0 & 0 & 0 & 0 & 1 & 0 \end{pmatrix}\)} \\
		\(4\) & \(\U \oplus [-2]^{\oplus2}  \oplus [4]\) & \([2] \oplus [-2]\oplus [-4]\) &   {\tiny \(\begin{pmatrix} -1 & 0 & 0 & 0 & 0 & 0 & 0 & 0 \\ 0 & -1 & 0 & 0 & 0 & 0 & 0 & 0 \\ 0 & 0 & 0 & 1 & 0 & 0 & 0 & 0 \\ 0 & 0 & 1 & 0 & 0 & 0 & 0 & 0 \\ 0 & 0 & 0 & 0 & 0 & -1 & 0 & 0 \\ 0 & 0 & 0 & 0 & -1 & 0 & 0 & 0 \\ 0 & 0 & 0 & 0 & 0 & 0 & 0 & 1 \\ 0 & 0 & 0 & 0 & 0 & 0 & 1 & 0 \end{pmatrix}\)}  \\
		\(5\) & \(\U^{\oplus2} \oplus [-4]\) & \(\U \oplus [-4]\) &  {\tiny \(\begin{pmatrix} -1 & 0 & 0 & 0 & 0 & 0 & 0 & 0 \\ 0 & -1 & 0 & 0 & 0 & 0 & 0 & 0 \\ 0 & 0 & -1 & 0 & 0 & 0 & 0 & 0 \\ 0 & 0 & 0 & -1 & 0 & 0 & 0 & 0 \\ 0 & 0 & 0 & 0 & 1 & 0 & 0 & 0 \\ 0 & 0 & 0 & 0 & 0 & 1 & 0 & 0 \\ 0 & 0 & 0 & 0 & 0 & 0 & 0 & 1 \\ 0 & 0 & 0 & 0 & 0 & 0 & 1 & 0 \end{pmatrix}\)} \\
        \rowcolor{blue!20!white}\(6\) &\([2] \oplus [-2]^{\oplus2} \oplus [4]\) & \([2] \oplus [-2]^{\oplus2} \oplus [-4]\) &  {\tiny \(\begin{colorpmatrix} 0 & 0 & 1 & 0 & 0 & 0 & 0 & 0 \\ 0 & 0 & 0 & 1 & 0 & 0 & 0 & 0 \\ 1 & 0 & 0 & 0 & 0 & 0 & 0 & 0 \\ 0 & 1 & 0 & 0 & 0 & 0 & 0 & 0 \\ 0 & 0 & 0 & 0 & 0 & -1 & 0 & 0 \\ 0 & 0 & 0 & 0 & -1 & 0 & 0 & 0 \\ 0 & 0 & 0 & 0 & 0 & 0 & 0 & 1 \\ 0 & 0 & 0 & 0 & 0 & 0 & 1 & 0 \end{colorpmatrix}\)} \\
	   \(7\) & \(\U \oplus[-2] \oplus [4]\) & \(\U \oplus [-2] \oplus [-4]\) & {\tiny \(\begin{pmatrix} 1 & 0 & 0 & 0 & 0 & 0 & 0 & 0 \\ 0 & 1 & 0 & 0 & 0 & 0 & 0 & 0 \\ 0 & 0 & -1 & 0 & 0 & 0 & 0 & 0 \\ 0 & 0 & 0 & -1 & 0 & 0 & 0 & 0 \\ 0 & 0 & 0 & 0 & 0 & -1 & 0 & 0 \\ 0 & 0 & 0 & 0 & -1 & 0 & 0 & 0 \\ 0 & 0 & 0 & 0 & 0 & 0 & 0 & 1 \\ 0 & 0 & 0 & 0 & 0 & 0 & 1 & 0 \end{pmatrix}\)} \\
       \rowcolor{blue!20!white}\(8\) & \(\U(2) \oplus [4]\) & \(\U(2) \oplus \A_3(-1)\) & {\tiny \(\begin{colorpmatrix} 0 & 0 & 1 & 0 & 0 & 0 & 0 & 0 \\ 0 & 0 & 0 & 1 & 0 & 0 & 0 & 0 \\ 1 & 0 & 0 & 0 & 0 & 0 & 0 & 0 \\ 0 & 1 & 0 & 0 & 0 & 0 & 0 & 0  \\ 0 & 0 & 0 & 0 & -1 & -2 & -1 & -1 \\ 0 & 0 & 0 & 0 & -2 & -1 & -1 & -1 \\ 0 & 0 & 0 & 0 & 2 & 2 & 2 & 1 \\ 0 & 0 & 0 & 0 & 2 & 2 & 1 & 2 \end{colorpmatrix}\)} \\
	    \(9\) & \([2]^{\oplus2} \oplus [-4]\) & \(\U\oplus [-2]^{\oplus2} \oplus [-4]\) & {\tiny \(\begin{pmatrix} 1 & 0 & 0 & 0 & 0 & 0 & 0 & 0 \\ 0 & 1 & 0 & 0 & 0 & 0 & 0 & 0 \\ 0 & 0 & 0 & -1 & 0 & 0 & 0 & 0 \\ 0 & 0 & -1 & 0 & 0 & 0 & 0 & 0 \\ 0 & 0 & 0 & 0 & 0 & -1 & 0 & 0 \\ 0 & 0 & 0 & 0 & -1 & 0 & 0 & 0 \\ 0 & 0 & 0 & 0 & 0 & 0 & 0 & 1 \\ 0 & 0 & 0 & 0 & 0 & 0 & 1 & 0 \end{pmatrix}\)} \\
		\(10\) & \(\U \oplus [4]\) & \(\U \oplus \A_{3}(-1)\) &  {\tiny \(\begin{pmatrix} 1 & 0 & 0 & 0 & 0 & 0 & 0 & 0 \\ 0 & 1 & 0 & 0 & 0 & 0 & 0 & 0 \\ 0 & 0 & -1 & 0 & 0 & 0 & 0 & 0 \\ 0 & 0 & 0 & -1 & 0 & 0 & 0 & 0 \\ 0 & 0 & 0 & 0 & -1 & -2 & -1 & -1 \\ 0 & 0 & 0 & 0 & -2 & -1 & -1 & -1 \\ 0 & 0 & 0 & 0 & 2 & 2 & 2 & 1 \\ 0 & 0 & 0 & 0 & 2 & 2 & 1 & 2 \end{pmatrix}\)} \\
        \rowcolor{blue!20!white}\(11\) & \([2] \oplus [4]\) & \(\U\oplus \A_{3}(-1) \oplus [-2]\) &  {\tiny \(\begin{colorpmatrix} 1 & 0 & 0 & 0 & 0 & 0 & 0 & 0 \\ 0 & 1 & 0 & 0 & 0 & 0 & 0 & 0 \\ 0 & 0 & 0 & -1 & 0 & 0 & 0 & 0 \\ 0 & 0 & -1 & 0 & 0 & 0 & 0 & 0 \\ 0 & 0 & 0 & 0 & -1 & -2 & -1 & -1 \\ 0 & 0 & 0 & 0 & -2 & -1 & -1 & -1 \\ 0 & 0 & 0 & 0 & 2 & 2 & 2 & 1 \\ 0 & 0 & 0 & 0 & 2 & 2 & 1 & 2 \end{colorpmatrix}\)} \\
\end{longtable}
\end{center}

\begin{center}
	\begin{longtable}{lllllr}
		
	\caption{Invariant and coinvariant sublattices of effective nonsymplectic groups \(G \subset \bO(\bL)\) on manifolds of \(\OG6\) type.}
		\label{tab:L} \\
		
		\toprule
	No. & \(|G|\) & \(|G^{\sharp}|\) & \(\bL_G\) & \(\bL^G\) & number of classes \\

	\midrule
	\endfirsthead
	
	\multicolumn{5}{c}%
	{\tablename\ \thetable{}, follows from previous page} \\
	\midrule
		No. & \(|G|\) & \(|G^{\sharp}|\) & \(\bL_G\) & \(\bL^G\) & number of classes \\
	\midrule
	\endhead
	
	\midrule
	\multicolumn{5}{c}{Continues on next page} \\
	\endfoot
	
	\bottomrule
	\endlastfoot
	
	\(1\) & \(2\) & \(1\) & \(\U^{\oplus 2} \oplus [ -2 ]^{\oplus 3 }\)& \([2]\) & 1\\
	\(2\) & \(2\)  & \(1\) & \(\U \oplus [2] \oplus [-2] ^{\oplus 3 }\) & \([2] \oplus [-2]\) & 1\\
   \rowcolor{blue!20!white}\(3\) & \(2\)  & \(1\) & \(\U \oplus [2] \oplus [-2] ^{\oplus 3 }\) & \(\U(2)\) & 1\\
	\(4\) & \(2\) & \(1\) & \(\U^{\oplus 2} \oplus [-2 ] ^{\oplus 2 }\) & \(\U\) & 1\\
	\(5\) & \(2\) & \(1\)  & \(\U^{\oplus 2} \oplus [-2 ] ^{\oplus 2 }\) & \([ 2 ] \oplus [-2]\) & 1\\
	\(6\) & \(2\) & \(1\) & \(\U^{\oplus 2} \oplus [-2 ] ^{\oplus 2 }\) & \(\U(2)\) & 1\\
	\(7\) & \(2\) & \(1\)  & \([2]^{\oplus2} \oplus [-2]^{\oplus3}\) & \([2] \oplus[-2]^{\oplus2}\) & 1\\
	\(8\) & \(2\) & \(1\) & \(\U \oplus [-2] ^{\oplus 2 } \oplus [2]\) & \(\U \oplus [-2]\) & 1\\
	\(9\) & \(2\) & \(1\) & \(\U \oplus [-2] ^{\oplus 2 } \oplus [2]\) & \([2] \oplus [-2]^{\oplus 2}\)  & 3\\
	\(10\) & \(2\) & \(1\) & \(\U^{\oplus 2} \oplus [-2]\) & \([-2] ^{\oplus 2} \oplus [2]\) & 1\\
	\(11\) & \(2\) & \(1\) & \(\U^{\oplus 2} \oplus [-2]\) & \(\U \oplus [-2]\) & 1\\
	\(12\) & \(2\) & \(1\) & \([2] ^{\oplus 2} \oplus [-2] ^{\oplus 2}\) & \(\U \oplus [-2] ^{\oplus 2}\) & 1\\
  \(13\) & \(2\) & \(1\)  & \([2] ^{\oplus 2} \oplus[-2] ^{\oplus 2}\) & \([-2] ^{\oplus 3} \oplus [2]\)  & 2\\
	\(14\) & \(2\) & \(1\) & \(\U(2)^{\oplus 2}\)  & \(\U(2) \oplus [-2]^{\oplus2}\) & 1\\
			\(15\) & \(2\) & \(1\)  & \(\U \oplus [2] \oplus [-2]\) & \([2] \oplus [-2]^{\oplus 3}\)  & 1\\
			\(16\) & \(2\) & \(1\) & \(\U \oplus [2] \oplus [-2]\) & \(\U \oplus [-2] ^{\oplus 2}\) & 1\\
			
			\(17\) & \(2\) & \(1\) & \(\U \oplus \U(2)\) & \(\U(2) \oplus [-2] ^{\oplus 2}\)  & 1\\
			\(18\) & \(2\) & \(1\) & \(\U \oplus \U(2)\) & \(\U \oplus [-2] ^{ \oplus 2}\) & 1 \\
			
			\(19\) & \(2\) & \(1\) & \(\U ^{\oplus 2}\)& \(\U \oplus [-2] ^{\oplus 2}\) & 1 \\
			
			\(20\) & \(2\) & \(1\)  & \([2]^{\oplus 2} \oplus [-2]\) & \([-2]^{\oplus 4} \oplus [2]\) & 1\\
			\(21\) & \(2\) & \(1\) & \([2]^{\oplus 2} \oplus [-2]\) & \(\U \oplus [-2]^{\oplus 3}\)& 2 \\
			
			\(22\) & \(2\) & \(1\) & \(\U \oplus [2]\) & \(\U \oplus [-2] ^{\oplus 3}\) & 1 \\
			
			\(23\) & \(2\) & \(1\) & \([2]^{ \oplus 2}\)& \(\U \oplus[-2] ^{\oplus 4}\) & 1\\
			\(24\) & \(2\) & \(1\) & \([2]^{ \oplus 2}\)& \(\U \oplus \bD_{4}(-1)\) & 1 \\
            \rowcolor{red!20!white}$25$ & \(2\)  & \(1\) & \([2]^{\oplus2}\) & \(\U(2) \oplus \bD_{4}(-1)\) &0 \\
			\midrule 

            \rowcolor{blue!20!white}\(1\) & \(2\)  & \(2\) & \(\U^{\oplus2}\oplus \A_3(-1)\) & \([4]\) & 1\\
            \rowcolor{blue!20!white}\(2\) & \(2\)  & \(2\) & \(\U\oplus [2]\oplus \A_3(-1)\) & \([4]\oplus [-2]\) & 1\\
    		\(3\) & \(2\)  & \(2\) & \(\U \oplus \U(2) \oplus [-4]\) & \(\U(2) \oplus [-4]\)& 1 \\
			\(4\) & \(2\)  & \(2\) & \(\U \oplus [2] \oplus [-2] \oplus [-4]\) & \([2] \oplus [-2] \oplus [-4]\)& 1 \\
			\(5\) & \(2\)  & \(2\) & \(\U^{\oplus2} \oplus [-4]\) & \(\U \oplus [-4]\) & 1\\
            \rowcolor{blue!20!white}\(6\) & \(2\)  & \(2\) & \([2] \oplus [-2]^{\oplus2} \oplus [4]\) & \([2] \oplus [-2]^{\oplus2} \oplus [-4]\)& 1\\
			\(7\) & \(2\)  & \(2\) & \(\U \oplus [2] \oplus [-4]\) & \(\U \oplus [-2] \oplus [-4]\) & 1\\
            \rowcolor{blue!20!white}\(8\) & \(2\)  & \(2\) & \(\U(2)\oplus [4]\) & \(\U(2)\oplus A_3(-1)\)& 1 \\
			\(9\) & \(2\)  & \(2\) & \([2]^{\oplus2}\oplus [-4]\) & \(\U \oplus [-2]^{\oplus2} \oplus [-4]\)& 1 \\
			\(10\) & \(2\)  & \(2\) & \(\U \oplus [4]\) & \(\U \oplus \A_3(-1)\)& 1 \\
            \rowcolor{blue!20!white}\(11\) & \(2\)  & \(2\) & \([2]\oplus [4]\) & \(\U\oplus \A_3(-1)\oplus [-2]\)& 1 \\
			\midrule
			
			\(1\) & \(3\) & \(1\) & \(\U^{\oplus2} \oplus \A_{2}(-1)\) & \([-2] \oplus [6]\) & 1\\
			\(2\)  & \(3\) & \(1\) & \(\U^{\oplus2}\) & \(\U \oplus[-2]^{\oplus2}\) & 1\\
			\(3\) & \(3\) & \(1\) & \(\U \oplus \U(3)\) & \(\U(3) \oplus [-2]^{\oplus2}\)& 1 \\
			\(4\) & \(3\) & \(1\) & \(\A_{2}\) & \(\U \oplus \A_{2}(-1) \oplus [-2]^{\oplus 2}\) & 1\\
			\midrule

			\(1\) & \(5\) & \(1\) & \(\U \oplus \bH_{5}\) & \([-2] \oplus [-10] \oplus \U\) & 1\\
			\midrule
			
			\(1\) & \(7\) & \(1\) & \(\U^{\oplus 2} \oplus \bK_{7}\) & \([-2] \oplus [14]\) & 1 \\
			
	\end{longtable}
\end{center}

\newcommand{\etalchar}[1]{$^{#1}$}
\bibliographystyle{abbrv}
\bibliography{main}

\medskip \medskip

\noindent Annalisa Grossi\\
\noindent {Alma Mater Studiorum - Università di Bologna, Dipartimento di Matematica, Piazza 
	di porta San Donato 5, 40126 Bologna}\\
\medskip \noindent{\texttt{annalisa.grossi3@unibo.it}}

\noindent{Stevell Muller}\\
\noindent{Institut für Algebraische Geometrie, Leibniz Universität Hannover, 30167 Hannover (Germany)}\\
\medskip \noindent{\texttt{muller@math.uni-hannover.de}}

\end{document}